\documentclass[10pt]{amsart}

\usepackage{amsmath,amssymb,amsthm}
\usepackage[left=1.1875in,top=1in,text={6.5in,9in}]{geometry}

\newtheorem{theorem}{Theorem}
\theoremstyle{plain}
\newtheorem{lemma}[theorem]{Lemma}
\newtheorem{prop}[theorem]{Proposition}

\theoremstyle{definition}
\newtheorem{definition}[theorem]{Definition}
\theoremstyle{remark}

\newcommand{\Z}{\mathbb{Z}}

\newcommand{\ds}{\displaystyle}

\begin{document}

\title{On the Homology of Elementary Abelian Groups as Modules over
  the Steenrod Algebra}

\author{Shaun V. Ault}

\author{William Singer}

\address{Department of Mathematics, Fordham University, Bronx, NY
  10458, U.S.A}

\begin{abstract}
  We examine the dual of the so-called "hit problem", the latter being
  the problem of determining a minimal generating set for the
  cohomology of products of infinite projective spaces as module over
  the Steenrod Algebra $\mathcal{A}$ at the prime 2.  The dual problem
  is to determine the set of $\mathcal {A}$-annihilated elements in
  homology.  The set of $\mathcal{A}$-annihilateds has been shown by
  David Anick to be a free associative algebra.  In this note we prove
  that, for each $k \geq 0$, the set of {\it $k$ partially
    $\mathcal{A}$-annihilateds}, the set of elements that are
  annihilated by $Sq^i$ for each $i\leq 2^k$, itself forms a free
  associative algebra. \end{abstract}

\maketitle

%%%%%%%%%%%%%%%%%%%%%%%%%%%%%%%%%%%%%%%%%%%%%%%%%%%%%%%%%%%%%%%%%%%%%%%%%%%%%%%%
\section{Introduction and Notations}\label{sec.intro}
%%%%%%%%%%%%%%%%%%%%%%%%%%%%%%%%%%%%%%%%%%%%%%%%%%%%%%%%%%%%%%%%%%%%%%%%%%%%%%%%

Let $\mathbb{F}_2$ be the field of $2$ elements and $\Gamma = \{
\Gamma_{s, *} \}_{s \geq 0}$ be the bigraded $\mathbb{F}_2$-space
defined by
\[
  \Gamma_{s, *} = H_{*}(B(\Z / 2)^{\times s}, \mathbb{F}_2), \quad
  \textrm{ for each $s \geq 0$.}
\]
The bigrading $(s, d)$ is by number of direct product factors of $B(\Z
/ 2)$ and homological degree.  We shall say that an element $x$ of
bidegree $(s, d)$ has {\it length} $s$ and {\it degree} $d$.  This
paper studies $\Gamma$ with its canonical structure as a right module
over the Steenrod algebra $\mathcal{A}$.  We are interested in
particular in the problem of determining the graded vector spaces
\[
  \Gamma_{s,*}^{\mathcal{A}} = \{ a \in \Gamma_{s,*} \;|\; (a)Sq^k =
  0, \forall k > 0\}
\]
consisting of all elements of $\Gamma$ that are annihilated by the
Steenrod operations of positive degree.  This problem and its dual
(finding a minimal generating set for the cohomology of $B(\Z /
2)^{\times s}$ as a left $\mathcal{A}$-module) have been much studied
in recent years.  The reader can find comprehensive bibliographies
covering work done through 2000 in \cite{Wood98} and \cite{Wood00},
and will find recent work in \cite{Kameko03}, \cite{Nam04} and
\cite{Sum07}.  Much progress has been made, but the general problem
remains unsolved.

We will work in terms of the reduced homology groups of the smash
products:
\[
  \widetilde{\Gamma}_{s, *} = \widetilde{H}_{*}(B(\Z / 2)^{\wedge s},
  \mathbb{F}_2), \quad \textrm{ for each $s \geq 1$.}
\]
We adopt also the convention $\widetilde{\Gamma}_{0, *} = H_*(\ast,
\mathbb{F}_2)$, the homology of a point.  We assemble the spaces
$\widetilde{\Gamma}_{s, *}$ into a bigraded vector space
\[
  \widetilde{\Gamma} = \{\widetilde{\Gamma}_{s, *} \}_{s \geq 0}
\]
and write the associated vector spaces of $\mathcal{A}$-annihilated
elements $\widetilde{\Gamma}_{s, *}^{\mathcal{A}} \subseteq
\widetilde{\Gamma}_{s, *}$.  The vector spaces $\Gamma_{s,
  *}^{\mathcal{A}}$ are easily expressed in terms of the spaces
$\widetilde{\Gamma}_{p, *}^{\mathcal{A}}$ for $p \leq s$, so a study
of the smash products is sufficient.  The natural mappings:
\[
  B(\Z / 2)^{\wedge p} \times B(\Z / 2)^{\wedge q} \longrightarrow
  B(\Z / 2)^{\wedge (p+q)}
\]
induce pairings of vector spaces:
\[
  \widetilde{\Gamma}_{p, *} \otimes \widetilde{\Gamma}_{q, *}
  \longrightarrow \widetilde{\Gamma}_{p+q, *}
\]
for all $p, q \geq 0$, which make $\widetilde{\Gamma}$ into a connected
bigraded algebra.  By the K\"{u}nneth theorem, $\widetilde{\Gamma}$ is
a free associative $\mathbb{F}_2$-algebra.  For each $k \geq 1$, it is
convenient to represent the canonical generators,
\[
  \gamma_k \in \widetilde{\Gamma}_{1,k} = \widetilde{H}_k(B(\Z/2)).
\]
Then we have $\widetilde{\Gamma} = \mathbb{F}_2 \langle \{ \gamma_1,
\gamma_2, \gamma_3, \ldots \} \rangle$ (we use the notation of
Cohn~\cite{C}: for a field $\Bbbk$ and set $X$, $\Bbbk\langle X
\rangle$ is the free associative $\Bbbk$-algebra generated by $X$).
The Cartan formula implies that the bigraded vector space
$\widetilde{\Gamma}^{\mathcal{A}} = \{\widetilde{\Gamma}_{s,
  *}^{\mathcal{A}}\}_{s\geq 0}$ is a subalgebra of
$\widetilde{\Gamma}$.  Anick proves in~\cite{A1} that this subalgebra
is itself free.  Now for each $k \geq 0$, and $s, d \geq 0$, define:
\[
  \Delta(k)_{s,d} = \bigcap_{i = 0}^k \mathrm{ker} (Sq^{2^i} :
  \widetilde{\Gamma}_{s,d} \to \widetilde{\Gamma}_{s, d-2^i}),
\]
and set $\Delta(k) = \{ \Delta(k)_{s,d} \}_{s, d \geq 0}$, a bigraded
space called the ``$k$ partially $\mathcal{A}$-annihilateds.''  Using
a variant of Anick's argument, we will show that

\begin{theorem}
  For each $k \geq 0$, $\Delta(k)$ is a free subalgebra of
  $\widetilde{\Gamma}$.
\end{theorem}

Note that if $k$ is chosen so that $d < 2^{k+2}$, then we have for any
$s \geq 0$:
\[
  \widetilde{\Gamma}^{\mathcal{A}}_{s,d} = \Delta(k)_{s,d}.
\]
Thus, determining the sets $S_k$ such that $\Delta(k) =
\mathbb{F}_2\langle S_k\rangle$ would solve the
``$\mathcal{A}$-annihilated problem'', and the solution would be in
terms of explicitly-given algebra generators.  Furthermore, partial
progress is meaningful, as the determination of the set $S_k$ would
determine all $\mathcal{A}$-annihilateds of degree $d < 2^{k+2}$.

In this note, $\mathrm{ker}Sq^{p}$ and $\mathrm{im}Sq^p$ will be
understood to involve the restricted maps $Sq^p : \widetilde{\Gamma}
\to \widetilde{\Gamma}$.

%%%%%%%%%%%%%%%%%%%%%%%%%%%%%%%%%%%%%%%%%%%%%%%%%%%%%%%%%%%%%%%%%%%%%%%%%%%%%%%%
\section{Proof of the Main Theorem}
%%%%%%%%%%%%%%%%%%%%%%%%%%%%%%%%%%%%%%%%%%%%%%%%%%%%%%%%%%%%%%%%%%%%%%%%%%%%%%%%

In this section we prove that $\Delta(k)$ is a tensor algebra, using a
remarkable lemma of Anick~\cite{A1}, here stated for the case of
$\Z^{t}$-graded algebras. Let $\Bbbk$ be a field.  For $t \geq 1$, a
$\Bbbk$-algebra $A$ is {\it $\Z^{t}$-graded} if $A = \{ A_{I} \}_{I
  \in \Z^{t}}$, and multiplication in $A$ is a family of maps,
\[
  A_{I} \otimes A_{J} \to A_{I+J}.
\]
If $x \in A_{I}$, we say the degree of $x$ is $I = (i_1, i_2, \ldots,
i_t)$.  Introduce a lexicographic ordering of degree as follows: $I <
J$ if and only if there is an integer $r$ with $1 \leq r \leq t$ such
that $i_p = j_p$ if $p < r$, and $i_{r} < j_{r}$.  The algebra $A$ is
said to be connected if $A_I = 0$ whenever $I$ contains a negative
entry, and $A_{(0,\ldots, 0)} \cong \Bbbk$.  The {\it
  positively-graded} elements of $A$ are the elements of the set
\[
  A^+ = \bigcup \{A_{I} \;|\; \textrm{all entries of $I$ are non-negative
    and at least one entry is positive}\}.
\]

\begin{definition}\label{def.anick}
  Let $A$ be a connected $\Z^t$-graded algebra ($t \geq 1$) over a
  field $\Bbbk$.  $A$ is said to satisfy {\it Anick's Condition} if
  whenever a relation,
  \begin{equation}\label{eqn.relation}
    \sum_{i=1}^n a_ib_i = 0,
  \end{equation}
  holds in $A$, where each $b_i \neq 0$, then there is a $j$ such that
  \begin{equation}\label{eqn.ideals}
    a_j \in \sum_{i \neq j} a_iA.
  \end{equation}
\end{definition}

\begin{lemma}[Anick~\cite{A1}]\label{lem.anick}
  Let $A$ be a connected $\Z^t$-graded algebra over a field $\Bbbk$.  Then
  $A$ is a tensor algebra, $A = \Bbbk\langle X \rangle$, for some set
  of positively-graded elements $X \subset A^+$, if and only if $A$
  satisfies Anick's Condition.
\end{lemma}

Anick's proof makes use of the work of Cohn~\cite{C} on so-called free
ideal rings (firs).  For completeness, we shall provide a proof that
avoids as much of this machinery as possible.  Furthermore, our
working in the graded case allows us to simplify some of Cohn's
arguments.

\begin{proof}
  The backward direction is the easier of the two.  The proof is
  Anick's~\cite{A1}.  Suppose that $A$ is not a tensor algebra.  Choose
  a minimal set $X$ of generators for $A$, and write $A = \Bbbk\langle
  X \rangle / R$ where $R$ is the non-zero ideal of relations.  Choose
  a non-zero $\alpha \in R$ of minimal degree.  Then $\alpha$ can be
  expanded uniquely in the form:
  \[
    \alpha = \sum_{i=1}^{m} x_iY_i,
  \]
  where the $x_i$ are distinct elements of the generating set $X$, and
  $Y_i \neq 0$ for each $i$.  For each $x \in \Bbbk\langle X \rangle$,
  write $\overline{x}$ for the corresponding element of $A$.  Since
  $\alpha \in R$, we have:
  \begin{equation}\label{eqn.rel}
    \sum_{i=1}^m \overline{x_i}\overline{Y_i} = 0.
  \end{equation}

  Since $\alpha$ is of minimal degree in $R$, we have
  $\overline{Y_i} \neq 0$ in $A$, for each value of $i$.  So
  Eqn.~(\ref{eqn.rel}) is a relation of the form (\ref{eqn.relation}).
  But if there were a $j$ with $1 \leq j \leq m$ such that
  \begin{equation}\label{eqn.overline_sum}
    \overline{x_j} \in \sum_{i \neq j} \overline{x_i}A,
  \end{equation}
  then the generating set $X$ would not be minimal, contradicting our
  assumption.  Thus $A$ cannot satisfy Anick's Condition.

  For the forward direction, assume that $A$ is a connected graded
  tensor algebra $\Bbbk\langle X \rangle$ on a generating set $X
  \subset A^+$.  Suppose now that there is a relation,
  \begin{equation}\label{eqn.sum}
    \sum_{i=1}^n a_ib_i = 0,
  \end{equation}
  for $a_i, b_i \in A$, and each $b_i \neq 0$, as in the premise of
  Anick's Condition.  We may assume the summands are ordered so that
  $deg(b_1) \geq deg(b_2) \geq \cdots \geq deg(b_n)$.  Let $I =
  deg(b_n)$, and let $c\mu = cx_1\cdots x_s$ be a term of degree $I$
  occurring in $b_n$ ($c \in \Bbbk, x_i \in X$).  For any element $a
  \in A$, we may write $a = a_0 + a^*\mu$ for some $a_0, a^* \in A$
  such that $\mu$ does not right-divide any term of $a_0$.  Moreover,
  both $a_0$ and $a^*$ are uniquely-determined since $A$ is free.
  Observe, the function $a \mapsto a^*$ is $\Bbbk$-linear of degree
  $-deg(\mu) = -I$ (This mapping is known as left transduction for
  $\mu$, and $a^*$ is known as the left cofactor of $\mu$ in
  $a$. See~\cite{C2}).

  Suppose $b \in A$ is any single term.  Then either $\mu$ does not
  right-divide $b$, in which case $b^* = 0$, or $\mu$ does, and
  $b^*\mu = b$.  Thus, if $deg(b) \geq I$, then for any $a \in A$,
  $(ab)^* = ab^*$.  
  By linearity of transduction, we have:
  \[
    (ab)^* = ab^*, \qquad \textrm{for any $a, b \in A$ such that
    $deg(b) \geq I$}.
  \]
  Applying transduction for $\mu$ to Eqn.(\ref{eqn.sum}), we have,
  since each $b_i$ has degree at least $I$, 
  \[
    0 = \left( \sum_{i=1}^n a_ib_i \right)^* = \sum_{i=1}^n a_ib_i^*
  \]
  Finally, since $b_n^* \neq 0$ has degree $(0, \ldots, 0)$, and $A$
  is connected, we have in fact shown that $b_n^* \in
  \Bbbk^{\times}$. We obtain a relation of the form:
  \[
    a_n = \left(-\sum_{i=1}^{n-1} a_i b_i^*\right)(b_n^*)^{-1}
    = \sum_{i=1}^{n-1} a_i(-b_n^*)^{-1}b_i^*
  \]
  Therefore, $\ds{ a_n \in \sum_{i \neq n} a_iA }$, as desired.
\end{proof}

The following is a useful application of Lemma \ref{lem.anick}.
\begin{lemma}\label{lem.any}
 Let $A$ be a connected $\Z^t$-graded algebra over a field $\Bbbk$,
 and suppose that $A$ is a tensor algebra, $A = \Bbbk\langle X
 \rangle$, on some set of positively-graded elements $X \subset A^+$.
 Let $S$ be {\it any} set of positively graded elements that form a
 minimal generating set for $A$.  Then $A$ is the tensor algebra on
 $S$.
\end{lemma}
\begin{proof}
Consider the canonical algebra mapping $\Bbbk\langle S \rangle\to A$.
We must show that the kernel is zero.  Suppose to the contrary there
is a non-zero element of the kernel.  We choose one of least degree;
say $\sum_{i=1}^n s_iY_i$, where the elements $s_i$ are distinct
members of the set $S$, and each $Y_i$ is a non-zero element of
$\Bbbk\langle S \rangle$.  Then we have in $A$ the relation:
\begin{equation}\label{relation}
\sum_{i = 1}^n \overline{s_i}\overline{Y_i} = 0.
\end{equation}
 Our assumption that \eqref{relation} is a relation of least degree
 assures that each $\overline{Y_i}$ is a non-zero element of $A$ .
 But by Lemma \ref{lem.anick}, A satisfies Anick's condition.  So
 there must be an index $j$ with $1\leq j\leq n$ and elements $c_i\in
 A$ for $i \neq j$ such that:
\begin{equation}\label{dependence}
\overline{s_j} = \sum_{i\neq j}\overline{ s_i}c_i.
\end{equation}
Now for each index $i\neq j$, the element $c_i$ must be expressible as
a non-commutative polynomial in the elements of $S$.  Further, since
$A$ is graded-connected, and each $\overline{s_i}$ has positive
degree, none of these polynomials can involve the element
$\overline{s_j}$.  Hence, equation \eqref{dependence} expresses
$\overline{s_j}$ in terms of the other generators.  This dependence
would contradict the assumed minimality of the generating set $S$.
Thus there can be no relation of the form \eqref{relation}, and the
result is proved.
\end{proof}

Now we come to our main result.

\begin{theorem}\label{thm.free}
  For $k \geq 0$, $\Delta(k)$ is free as associative
  $\mathbb{F}_2$-algebra.
\end{theorem}
\begin{proof}
  We will show that $\Delta(k)$ satisfies Anick's Condition.
  Suppose there is a relation in $\Delta(k)$,
  \begin{equation}\label{eqn.Delta_k_sum}
    \sum_{i=1}^n a_ib_i = 0,
  \end{equation}
  where each $b_i \neq 0$.  We want to show that there is an
  index $j$ such that
  \[
    a_j \in \sum_{i \neq j} a_i \Delta(k).
  \]
  This will surely be the case if the elements $a_1, \ldots, a_n$ were
  not distinct, so we may assume that the $a_i$ are distinct.  Now
  Eqn.~(\ref{eqn.Delta_k_sum}) can be read as a relation in the
  connected tensor algebra $\widetilde{\Gamma}$.  The fact that one
  such relation among the elements of $\{a_i\}$ exists means that we
  can find one for which $n$ is minimal.  In
  other words, let
  \begin{equation}\label{eqn.Delta_k_sum2}
    \sum_{i=1}^p a_ic_i = 0,
  \end{equation}
  be a relation with minimal number of summands in $\widetilde{\Gamma}$,
  involving elements from the set $\{a_i\}$, with each $c_i \neq 0$.
  Since $\widetilde{\Gamma}$ satisfies Anick's Condition (by
  Lemma~\ref{lem.anick}), there is an index $j$ such that
  \begin{equation}\label{eqn.Gamma_relation}
    a_j = \sum_{1 \leq i \leq p,\, i\neq j} a_i d_i,
  \end{equation}
  for some $d_i \in \widetilde{\Gamma}$.  We shall show that every
  $d_i$ is in fact a member of $\Delta(k)$.  Let $\ell$ be an integer,
  $0 \leq \ell \leq k$.  Apply $Sq^{2^\ell}$ to both sides of
  Eqn.~(\ref{eqn.Gamma_relation}).  Note, $a_iSq^q = 0$ for each $q$ satisfying $0
 <q \leq 2^{\ell}$ and every $i$, since $a_i \in \Delta(k)$.
  Hence, by the Cartan formula,
  \begin{equation}\label{eqn.Gamma_relation_Sq}
    0 = \sum_{i \neq j} a_i \left( d_i Sq^{2^\ell} \right).
  \end{equation}
  If there are any indices $i$ such that $d_i Sq^{2^\ell} \neq 0$, then
  Eqn.~(\ref{eqn.Gamma_relation_Sq}) would represent a non-trivial
  relation among the elements of $\{a_i\}$, of strictly fewer number
  of terms than the supposed minimal one.  Therefore, $d_i Sq^{2^\ell} =
  0$ for each $i$.  But this is true for any $0 \leq \ell \leq k$, so
  each $d_i \in \Delta(k)$.  This shows that $a_j \in \sum_{i \neq j}
  a_i\Delta(k)$, and so $\Delta(k)$ satisfies Anick's Condition.  Hence $\Delta(k)$
  is a tensor algebra on a positively-graded generating set.
\end{proof}

%%%%%%%%%%%%%%%%%%%%%%%%%%%%%%%%%%%%%%%%%%%%%%%%%%%%%%%%%%%%%%%%%%%%%%%%%%%%%%%%

%%%%%%%%%%%%%%%%% NEW %%%%%%%%%%%%%%%%%%%%%%%

\section{Analysis of $\Delta(0)$}

In an effort to de-clutter our formulas, we use the notation:
\begin{eqnarray*}
  [i_1, i_2, \ldots, i_s] &=& \gamma_{i_1}\gamma_{i_2}\cdots
  \gamma_{i_s} \in \widetilde{\Gamma}_{s, *},\, s \geq 1\\ 
        {[} \;{]} &=& 1 \in \widetilde{\Gamma}_{0,0}
\end{eqnarray*}

For $s \geq 1$, and integers $m_i \geq 0$, we define special elements of
$\widetilde{\Gamma}$: 
\[
  \sigma(m_1, m_2, \ldots, m_s) \stackrel{def}{=} [2m_1+2, 2m_2+2,
    \ldots, 2m_s+2]Sq^1,
\]
and we let $S_0$ be the set:
$$S_0\,\,=\,\,\{ \sigma(m_1, m_2, \ldots, m_s) \,|\, s\geq 1, m_1\geq 0, \ldots m_s\geq 0\}.$$

Our goal in this section is to prove that $\Delta(0) = \mathrm{ker}Sq^1$ is the free algebra on the set 
$S_0$.

\begin{lemma}\label{imker}
For each $s \geq 1$ one has in $\widetilde{\Gamma}_{s, \ast}$:
\[
  \mathrm{ker} Sq^1 = \mathrm{im}Sq^1.
\]
\end{lemma}

\begin{proof}
  $Sq^1$ acts as a differential on $\widetilde{\Gamma}_{s, \ast}$; and
  the isomorphism:
  \[
    \widetilde{\Gamma}_{s, \ast} = (\widetilde{\Gamma}_{1,
      \ast})^{\otimes s}
  \]
  is an isomorphism of chain complexes.  Since $\widetilde{\Gamma}_{1,
    \ast}$ is acyclic, our results follows from the K\"{u}nneth theorem.
\end{proof}

Given a monomial $\mu = [i_1, i_2, \ldots, i_s] $ in
$\widetilde{\Gamma}_{s,*}$ we define its weight to be the number of
the indices $i_1,i_2,\ldots i_s$ that are odd.
\begin{lemma}\label{genim}
  Let 
  $\mu\in\widetilde{\Gamma}$ be any monomial.  Then
  $(\mu)Sq^1$ lies in the algebra generated by $S_0$.
\end{lemma}
\begin{proof}
We will prove the lemma by induction on $t$, the weight of $\mu$.  The
case $t = 0$ is tautological.  Now suppose that $t\geq 1$ and that the
lemma has been proved for all monomials $\mu$ of weights less than
$t$.  Let $\mu = [i_1, i_2, \ldots, i_s]$ be a given monomial of
weight $t$.  Choose an index $i_k$ that is odd; say, $i_k = 2m - 1$.
Then $[i_k] = [2m]Sq^1$ and $[i_k]Sq^1 = 0$. Let $\alpha = [i_1,
  \ldots, i_{k-1}]$ and $\beta = [i_{k+1}, \ldots, i_s]$ so that using
the product in $\widetilde{\Gamma}$ we may write: $\mu =
\alpha\cdot[i_k]\cdot\beta$.  Then,
\begin{eqnarray*}
  (\mu)Sq^1 &=& (\alpha\cdot [i_k]\cdot\beta)Sq^1\\
  &=& (\alpha)Sq^1 \cdot[i_k]\cdot\beta + 0 + 
  \alpha\cdot [i_k]\cdot (\beta)Sq^1\\
  &=& (\alpha)Sq^1\cdot [i_k]\cdot\beta + 
  (\alpha)Sq^1 \cdot[2m]\cdot(\beta)Sq^1 + 
  (\alpha)Sq^1\cdot [2m]\cdot (\beta)Sq^1 + 
  \alpha\cdot[i_k]\cdot(\beta) Sq^1\\
  &=& (\alpha)Sq^1 \cdot([2m]\cdot\beta)Sq^1 + 
    (\alpha\cdot [2m])Sq^1\cdot (\beta)Sq^1.\\
\end{eqnarray*}
But the right hand side of this equation is a sum of products of
elements of the form $(\gamma)Sq^1$, where in each case, $\gamma$ is a
monomial of weight less than $t$.  So the inductive hypothesis implies
that $(\mu)Sq^1$ lies in the algebra generated by $S_0$, and our
inductive proof is complete.
\end{proof}

Combining Lemmas \ref{imker} and \ref{genim}, we find:
\begin{lemma}\label{genker}
  $\Delta(0)$ is generated as an algebra by the set $S_0$.
\end{lemma}

%%%%%%%%%%%%%%%%% END NEW %%%%%%%%%%%%%%%%%%%%
The next lemma will be useful in proving that the set $S_0$ is algebraically independent.
In what follows, when we write
``$\alpha$ expanded in terms of $\widetilde{\Gamma}$'', we mean to
express $\alpha$ as a sum of monomials $[i_1, i_2, \ldots, i_s] \in
\widetilde{\Gamma}_{s,*}$.  By abuse of notation, we say that $i_j$ is
a factor of the term $[i_1, i_2, \ldots, i_s]$, and so we may speak of
odd or even factors of such a term.

\begin{lemma}\label{lem.linind}
  $S_0$ is a linearly independent subset of $\widetilde{\Gamma}.$
\end{lemma}
\begin{proof}
  Since $\widetilde{\Gamma}$ is graded on length, 
  any potential linear dependence will be of the form
  \begin{equation}\label{eqn.lindep}
    \sum_{j} \sigma(m_{j,1}, m_{j,2}, \ldots, m_{j,s}) = 0.
  \end{equation}
  where the sum is over a positive number of indices $j$ and the value
  of $s$ is the same for all terms.  We wish to show that such a
  relation is impossible if the sequences $\{m_{j,1}, m_{j,2}, \ldots,
  m_{j,s}\}$ are all distinct.  But when a term $ \sigma(m_{j,1},
  m_{j,2}, \ldots, m_{j,s})$ is expanded in terms of
  $\widetilde{\Gamma}$, the monomial $[2m_{j,1} + 1, 2m_{j,2}+
    2,\ldots, 2m_{j,s}+2]$ will appear, and this monomial cannot
  appear in any of the other expansions.  So a relation of the form
  \eqref{eqn.lindep} is impossible.
  \end{proof}
  
  Preparing the proof of the next result, we define
  $W_k\widetilde{\Gamma}$ to be the subspace of $\widetilde{\Gamma}$
  that is spanned by all the monomials $[i_1, i_2, \ldots, i_s] $ of
  weight $k$.  Then as a bigraded vector space, $\widetilde{\Gamma}$
  splits as a direct sum:
\begin{equation}\label{directsum}
\widetilde{\Gamma} = \bigoplus_{k\geq  0} W_k\widetilde{\Gamma}.
\end{equation}
The product on $\widetilde{\Gamma}$ is compatible with this direct sum
decomposition, in the sense that:
\begin{equation}\label{product}
 W_k\widetilde{\Gamma}\,\cdot W_{\ell}\widetilde{\Gamma}\subseteq
 W_{k+\ell}\widetilde{\Gamma}.
\end{equation}

\begin{prop}\label{prop.tens_Delta0}
  $\Delta(0) = \mathbb{F}_2\langle S_0\rangle$.
\end{prop}
\begin{proof}
  We have already shown that the set of all products of elements from
  $S_0$ generates $\mathrm{ker}Sq^1$ as a vector space.  All that
  remains is to prove that $S_0$ is an algebraically independent set.
  According to Lemma \ref{lem.any}, it will be enough to prove that
  $S_0$ is a {\it minimal} generating set for the algebra
  $\mathrm{ker}Sq^1$.  Suppose the contrary.  Then it would be
  possible to express one particular generator, say, $ \sigma(m_1,
  m_2, \ldots, m_s)$, as a (non-commutative) polynomial in the other
  generators.  But each generator $ \sigma(n_1, n_2, \ldots, n_l)$
  lies in $W_1\widetilde{\Gamma}$, so equations \eqref{directsum} and
  \eqref{product} imply that our expression of $ \sigma(m_1, m_2,
  \ldots, m_s)$ in terms of the other generators would reduce to an
  expression of $ \sigma(m_1, m_2, \ldots, m_s)$ as a linear
  combination of the others.  According to Lemma \ref{lem.linind},
  this is not possible.  So $S_0$ must be a minimal generating set for
  the algebra $\mathrm{ker}Sq^1$, and our proof is complete.
   \end{proof}

Using Prop.~\ref{prop.tens_Delta0}, we can determine a formula
for the dimension of the homogeneous components of $\Delta(0)$.
\begin{prop}\label{prop.dim_Delta0}
  Let $c_{s,d}$ be the dimension of the component of $\Delta(0)$ in
  bidgree $(s, d)$.  Let $\eta_{s,d}$ be defined by
  \[
    \eta_{s,d} = \left\{\begin{array}{ll}
      0, & \textrm{if $d$ is even} \\
      \binom{\frac{d-1}{2}}{s-1}, & \textrm{if $d$ is odd}
        \end{array}\right.
  \]
  There is a recurrence relation,
  \begin{eqnarray*}
    c_{0,0} &=& 1 \\
    c_{s,0} &=& 0, \quad\textrm{if $s > 0$} \\
    c_{0,d} &=& 0, \quad\textrm{if $d > 0$} \\
    c_{s,d} &=& \sum_{r=1}^s \sum_{a=1}^d \eta_{r, a}c_{s-r, d-a}, \quad
      \textrm{if $s, d \geq 1$}\\
  \end{eqnarray*}
\end{prop}
\begin{proof}
  Since $\Delta(0)$ is a tensor algebra, $c_{s,d}$ counts the number of
  terms of length $s$ and degree $d$.  Partition the terms of
  $\Delta(0)_{s,d}$ by length and degree of the first factor in the
  term.  In what follows, ``$\sigma$'' always represents an
  algebra generator, such as $\sigma(m_1, m_2, \ldots, m_k)$, while
  ``$\tau$'' represents a (possibly empty) product of algebra generators.
  \[
    \Delta(0)_{s,d} = \mathrm{span}\left( \coprod_{r, a} \left\{
      \sigma \cdot \tau \;|\; \ell(\sigma) = r, deg(\sigma)= a,
      \ell(\tau) = s-r, deg(\tau) = d-a  \right\}\right)
  \]
  Now since $\tau$ is an arbitrary term in $\Delta(0)_{s-r, d-a}$,
  we obtain the formula,
  \[
     c_{s,d} = \sum_{r, a} (\textit{number of algebra generators
       $\sigma$ in bidegree $(r, a)$})\cdot c_{s-r, d-a}       
  \]
  For a typical $\sigma$ in bidegree $(r, a)$, we have $\sigma =
  \sigma(m_1, \ldots, m_r)$ such that $2(m_1 + \cdots + m_r) + 2r - 1
  = a$.  Thus $m_1 + \cdots + m_r = (a+1)/2 - r$.  So the number of
  algebra generators $\sigma$ in bidegree $(r, a)$ is found by
  counting the number of ordered partitions of $(a+1)/2 - r$ into $r$
  parts.  Elementary combinatorics tells us that the number of such
  partitions is exactly $\eta_{r, a}$ as defined above.

  The base cases for the recurrence are easily verified.
\end{proof}

As examples, we find closed formulas for $c_{s,d}$ for small $s$:
\begin{itemize}
  \item $\ds{c_{1,d} = \left\{\begin{array}{ll} 0, & \textrm{$d$ even}\\
          1, & \textrm{$d$ odd}\end{array}\right.}$
  \item $\ds{c_{2,d} = \left\{\begin{array}{ll} \frac{d}{2}, & 
          \textrm{$d$ even}\\ \frac{d-1}{2}, & \textrm{$d$ odd}
        \end{array}\right.}$
  \item $\ds{c_{3,d} = \left\{\begin{array}{ll}
          \frac{d(d-2)}{4}, & \textrm{$d$ even}\\
          \frac{(d-1)^2}{4}, & \textrm{$d$ odd}
        \end{array}\right.}$
\end{itemize}

It was pointed out by the referee that the exactness of $Sq^1$ on
$\widetilde{\Gamma}_{s,*}$ implies a nice reduction formula:
\begin{equation}
  c_{s,d} + c_{s,d+1} = dim\left( \widetilde{\Gamma}_{s,d+1} \right)
  = \binom{d}{s-1}.
\end{equation}

%%%%%%%%%%%%%%%%%%%%%%%%%%%%%%%%%%%%%%%%%%%%%%%%%%%%%%%%%%%%%%%%%%%%%%%%%%%%%%%%

\bibliographystyle{plain}

%\bibliography{refs}

\end{document}